\documentclass[11pt,twoside]{article}
\usepackage{amsmath}
\usepackage{amsthm}
\makeatletter
\renewenvironment{proof}[1][\proofname]{%
	\par\pushQED{\qed}\normalfont%
	\topsep6\p@\@plus6\p@\relax
	\trivlist\item[\hskip\labelsep\bfseries#1\@addpunct{.}]%
	\ignorespaces
}{%
	\popQED\endtrivlist\@endpefalse
}
\makeatother 
\usepackage{mathtools}
\usepackage[mathscr]{eucal}
\usepackage{amssymb}
\usepackage{graphicx}
\usepackage{enumerate}
\usepackage[authoryear,sort&compress,round,comma]{natbib}
\newtheorem{definition}{Definition}
\usepackage{mathptmx}

\newtheorem{corollary}{Corollary}[section]
\newtheorem{theorem}{Theorem}[section]
\newtheorem{lemma}{Lemma}[section]

\newtheorem{example}{Example}[section]
\numberwithin{equation}{section}

\usepackage{graphicx,type1cm,eso-pic,color}

\makeatletter
\AddToShipoutPicture{
	\setlength{\@tempdimb}{.5\paperwidth}
	\setlength{\@tempdimc}{.5\paperheight}
	\setlength{\unitlength}{1pt}
	\put(\strip@pt\@tempdimb,\strip@pt\@tempdimc)
}
\makeatother

\begin{document}

	\setcounter{page}{1}

	\thispagestyle{empty}
	\markboth{}{}

	\pagestyle{myheadings}
	\markboth{}{ }
	
	\date{}
	
	
	\noindent  
	
	\vspace{.1in}
	
	{\baselineskip 20truept
		
		\begin{center}
			{\Large {\bf On General Weighted Extropy of Percentile Ranked Set Sampling }} \footnote{\noindent
				{\bf $^{\#}$} E-mail: nitin.gupta@maths.iitkgp.ac.in\\
				{\bf * }  corresponding author E-mail: pradeep.maths@kgpian.iitkgp.ac.in}\\
			
		\end{center}
		
		\vspace{.1in}
		
		\begin{center}
			{\large {\bf Pradeep Kumar Sahu* and Nitin Gupta $^{\#}$}}\\
			{\large {\it Department of Mathematics, Indian Institute of Technology Kharagpur, West Bengal 721302, India }}
			\\
		\end{center}
	}
	\vspace{.1in}
	\baselineskip 12truept

	
	\begin{center}
		{\bf \large Abstract}\\
	\end{center}
	The extropy measure, first proposed by Lad, Sanfilippo, and Agro in their (2015) paper in Statistical Science, has attracted considerable attention in recent years. Our study introduces a fresh approach to representing weighted extropy in the framework of percentile ranked set sampling. Furthermore, we provide additional insights such as stochastic orders, characterizations, and bounds. Our findings illuminate the comparison between the weighted extropy of percentile ranked set sampling and its equivalent in simple random sampling.
	\section{Introduction}
	In (1952), McIntyre introduced ranked set sampling (RSS) as a method to estimate mean pasture yields, proving its superiority over simple random sampling (SRS) for population mean estimation.

Let us consider a random variable $X$ with probability density function (pdf) $f$, cumulative distribution function (cdf) $F$, and survival function (sf) $\bar F = 1 - F$. For a one-cycle RSS, we randomly select $n^2$ units from $X$, then allocate them randomly into $n$ sets of size $n$. Within each set, units are ranked based on the variable of interest. The selection process starts with choosing the smallest ranked unit from the first set, followed by the second smallest from the second set, and so on until the $n$th smallest ranked unit is selected from the last set. This process is repeated $m$ times to obtain a sample of size $mn$, referred to as the RSS from the distribution $F$. Denoting the ranked-set sample as $\mathbf{X}_{RSS}^{(n)}=\{X_{(i:n)i},\ i=1,\ldots,n\}$, where $X_{(i:n)i}$ represents the $i$-th order statistics from the $i$-th sample with a size of $n$.

For further understanding of RSS, refer to the works of Wolfe (2004), Chen, Bai, and Sinha (2004), Al-Nasser (2007), Al-Saleh and Diab (2009), and Raqab and Qiu (2019).

 In (1996), Samawi, Ahmed, and Abu-Dayyeh developed extreme ranked set sampling (ERSS) to improve the efficiency of estimating population parameters compared to SRS, using an equivalent number of units. ERSS accomplishes this by considering only the minimum and maximum ranked units when 
$n$ is even, and including the minimum, maximum, and median ranked units when $n$ is odd..

For even $n$, ERSS selects units for measurement from the smallest ranked units in the first $\frac{n}{2}$ sets and the largest ranked units in the last $\frac{n}{2}$ sets. In the case of odd $n$, units are selected from the smallest ranked units in the first $\frac{n-1}{2}$ sets, the largest ranked units in the subsequent $\frac{n-1}{2}$ sets, and the median value of the last set in the sample corresponds to the last unit.

 Muttlak (2003) proposes Percentile Ranked Set Sampling (PRSS) as a method to more efficiently estimate population parameters compared to SRS using an identical number of units. PRSS achieves this by solely utilizing the $[p(n+1)]$-th rank for first $\frac{n}{2}$ sets and $[q(n+1)]$-th ranked units for the next $\frac{n}{2}$ when $n$ is even and,  $[p(n+1)]$-th rank for first $\frac{n-1}{2}$ sets and $[q(n+1)]$-th ranked units for next $\frac{n-1}{2}$ sets along with the median ranked unit fom the last unit when $n$ is odd. In this sampling approach, $p\in (0,1)$ represents the percentile value, $q=1-p$; and $[p(n+1)]$ and $[q(n+1)]$ are rounded to the nearest integer. The cycle may be repeated $m$ times to get $nm$ units. These $nm$ units form the PRSS data.
 

In (2023), Akdeniz and Yildiz explored how ranking error models impact mean estimators derived from RSS and its variants like ERSS and PRSS, across various distributions, set sizes, and cycle sizes within infinite populations. They conducted a Monte Carlo simulation study supplemented by real-life data, revealing that RSS, ERSS, and PRSS outperform Simple Random Sampling (SRS).
	
	In (1948), Shannon introduced entropy, a concept applicable across multiple fields including information theory, physics, probability and statistics, economics, and communication theory. Entropy measures the average level of uncertainty associated with the results of a random experiment.
	
	Entropy measures the average level of uncertainty linked to the results of a random experiment. The differential expression for Shannon entropy is as follows:
	\begin{equation*}
		\label{1eq2}
		H(X) = E\left(-\ln f(X)\right) = - {\int_{-\infty}^{\infty}{f(x) \ln \left( f(x)\right)dx}}.
	\end{equation*} 
	This measure finds applications in diverse contexts, such as order statistics and record statistics. Its usage is exemplified in the studies of Bratpour et al. (2007), Raqab and Awad (2000, 2001), Zarezadeh and Asadi (2010), Abo-Eleneen (2011), Qiu and Jia (2018a, 2018b), and Tahmasebi et al. (2016).
	
	In modern times, another measure for evaluating uncertainty, called extropy, has gained prominence. Lad et al. (2015) presented extropy as the complementary counterpart of Shannon entropy, defining it as:
	\begin{equation}\label{extropy}
		J(X)=-\frac{1}{2} \int_{-\infty}^{\infty}f^2(x)dx=-\frac{1}{2}E\left(f(X)\right).
	\end{equation}

	In a study led by Qiu (2017), an examination of characterization results, lower limits, monotonic properties, and statistical uses related to the extropy of order statistics and record values was undertaken. Separately, Balakrishnan et al. (2020) and Bansal and Gupta (2021) introduced the idea of weighted extropy as:
	\begin{equation}\label{wtdextropy}
		J^w(X)=-\frac{1}{2}\int_{-\infty}^{\infty}xf^2(x)dx.
	\end{equation}
	
	In their study, Balakrishnan et al. (2020) explored characterization results and limits for weighted versions of extropy iterations, residual extropy, past extropy, bivariate extropy, and bivariate weighted extropy. Meanwhile, Bansal and Gupta (2021) investigated the outcomes associated with weighted extropy and weighted residual extropy concerning order statistics and and k-record values. In this framework, the notion of general weighted extropy (GWE) with a non-negative weight function $w_1(x)\geq 0$ is defined as:
	\begin{align*}
		J^{w_1}(X)&=-\frac{1}{2} \int_{-\infty}^{\infty}w_1(x) f^2(x)dx\\
		&=-\frac{1}{2}E(\Lambda_X^{w_1}(U)),
	\end{align*}
	where $\Lambda_X^{w_1}(u)=w_1(F^{-1}(u))f(F^{-1}(u))$ and $U$ is uniformly distributed random variable on $(0,1)$, i.e., $U\sim$ Uniform$(0,1)$.
	
	An illustration provided by Bansal and Gupta (2021) and Gupta and Chaudhary (2023) demonstrates that while the extropies may be similar, there exists a difference between weighted extropies and general weighted extropies. Thus, both weighted extropies and general weighted extropies act as indicators of uncertainty. Unlike the extropy defined in Equation (\ref{extropy}), this measure, influenced by the shift, takes into account the values of the random variable.
	
	In their research, Qiu and Raqab (2022) introduced a formulation for the 
 weighted extropy linked to RSS, articulated in terms of quantile and density-quantile functions. They also provided relevant findings, covering monotonic properties, stochastic orders, characterizations, and precise boundaries. Additionally, the authors conducted a comparative examination between the weighted extropy of RSS and its counterpart in SRS. In their study, Gupta and Chaudhary (2023) examined the monotonic and stochastic traits of the GWE concerning RSS data. The investigation involved deriving stochastic comparison outcomes through diverse weightings for RSS extropy and included a comparative assessment with SRS data. The study also yielded characterization findings and explored the monotonic properties of the GWE associated with RSS data. Recently Sahu and Gupta (2024) study the stochastic properties of GWE
of ERSS data. By taking different weights for the extropy of ERSS they obtain stochastic comparison results. They also provide some characterization results.

         This manuscript focuses on investigating the stochastic characteristics of the GWE concerning PRSS data. Stochastic comparison outcomes are derived by employing various weights for the extropy of PRSS. A comparison between the extropy of PRSS and SRS data is conducted. Additionally, certain characterization findings are obtained.

	\section{Some results on GWE and related properties}
	Before presenting findings, let's examine definitions from scholarly sources (refer to Shaked and Shanthikumar (2007) and Misra, Gupta, and Dhariyal (2008)) of pertinent terminology.
	
	\begin{definition}
		Let a random variable $X$ have the probability density function (pdf) $f(x)$, the cummulative distribution function (cdf) $F(x)$ and the survival function (sf) $\bar{F}(x) = 1-F(x).$ Let $l_X = inf\{x \in \mathbb{R} : F(x) > 0\},\   u_X = sup\{x \in \mathbb{R} : F(x) < 1\}$ and	 $S_X= (l_X,u_X),$ where -$\infty \le l_X \le u_X \le \infty$.\\
		\\
		(i) $X$ is said to be log-concave (log-convex) if $\{x \in \mathbb{R} : f(x) > 0\} = S_X$
		and $ln(f(x))$ is concave (convex) on $S _X$.\\
		\\
		(ii) $X$ is said to have increasing (decreasing) failure rate IFR (DFR) if $\bar F(x)$ is log-concave		(log-convex) on  $S_X$.\\
		\\
		(iii) $X$ is said to have decreasing (increasing) reverse failure rate DRFR (IRFR) if $F(x)$ is	log-concave (log-convex) on $S_X$.\\
		\\
		(iv) $X$ is said to have decreasing (increasing) mean residual life DMRL (IMRL) if $\int_{x}^{u_X} \bar F(t)dt$		is log-concave (log-convex) on $S_X$.\\
		\\
		(v) $X$ is said to have increasing (decreasing) mean inactivity time (IMIT (DMIT)) if $\int_{l_X}^{x}	F(t)dt$
		is log-concave (log-convex) on $S_X.$
		
	\end{definition}

	\begin{definition}
		Let $X$ be a random variable with  pdf $f(x)$, cdf $F(x)$ and sf $\bar{F}(x)=1-F(x).$ Let $l_X = inf\{x \in \mathbb{R} : F(x) > 0\}$, $u_X = sup\{x \in \mathbb{R} : F(x) < 1\}$
		and $S_X = (l_X, u_X).$ Similarly, let $Y$ be a random variable with  pdf $g(x)$, cdf $G(x)$ and sf $\bar{G}(x)= 1-G(x).$ Let $l_Y=inf\{x \in R:G(x)> 0\},\  u_Y = sup\{x \in R :G(x) < 1\}$ and $S_Y = (l_Y,u_Y )$. If $l_X \ge 0$ and $l_Y \ge 0$, then \\
		\\
		(i) $X$ is said to be smaller than $Y$ in usual stochastic (st) ordering		$(X \le_{st} Y )$ if $\bar{F}(x) \le \bar{G}(x)$, for every -$\infty< x <\infty.$\\
		\\
		(ii) $X$ is said to be smaller than $Y$ in the likelihood ratio (lr) ordering $(X \le_{lr} Y )$ if
		$g(x)f(y) \le  f(x)g(y)$, whenever $-\infty < x < y <\infty$.\\
		\\
		(iv) $X$ is said to be smaller than $Y$ in the dispersive ordering ($X \le_{disp} Y)$ if $G^{-1}F(x)-x$ is increasing in $x \ge 0$.\\
		\\
		(v) $X$ is said to be smaller than $Y$ in the hazard rate ordering $(X \le_{hr} Y)$ if $\frac{\bar{G}(x)}{\bar{F}(x)}$  is increasing in $x\in S_X \cap S_Y.$
		
	\end{definition}

\section{\textbf{General Weighted Extropies of PRSS}}	
	Let $X$ be a random variable with finite mean $\mu$ and variance $\sigma^2$. For $\textbf{X}_{SRS}=\{X_i,\ i=1,\ldots,n\}$, the joint pdf is $\prod_{i=1}^{n}f(x_i)$, as $X_i$'s, $i=1,\ldots,n$ are independent and identically distributed (i.i.d.). Hence the general weighted extropy (GWE) of $\textbf{X}_{SRS}^{(n)}$ can be defined as
	\begin{align}\label{SRS1}
		J^{w_1}(\textbf{X}_{SRS}^{(n)})&=\frac{-1}{2}\prod_{i=1}^{n}\left(\int_{-\infty}^{\infty}w_1(x_i)f^2(x_i)dx_i\right)\nonumber \\
		&=\frac{-1}{2}\left(-2J^{w_1}(X)\right)^n\nonumber \\
		&=\frac{-1}{2}\left(E(\Lambda_X^{w_1}(U))\right)^n.
	\end{align}
 Let the pdf of $X_{(i,n)i}$ is
      \[f_{i:n}(x) =\frac{n!}{(i-1)!(n-i)!}F^{i-1}(x) \bar F^{n-i}(x) f(x),\  \infty<x<\infty.\]
      	Now, we can write the GWE of   $\textbf{X}_{PRSS}^{(n)}$  as
\begin{align}	
	& J^{w_1}(\textbf{X}_{PRSS}^{(n)})\nonumber \\ 
 &=\begin{cases}
-\frac{1}{2}\left(\prod_{i=1}^{n/2}\left(-2J^{w_1}(X_{([p(n+1)]:n)i})\right)\right)\left(\prod_{i=\frac{n}{2}+1}^{n}\left(-2J^{w_1}(X_{([q(n+1)]+1:n)i})\right) \right) \text{ if } \mbox{n\  even},\\
-\frac{1}{2}\left(\prod_{i=1}^{\frac{n-1}{2}}\left(-2J^{w_1}(X_{([p(n+1)]:n)i})\right)\right)\left(\prod_{i=\frac{n+1}{2}}^{n-1}\left(-2J^{w_1}(X_{([q(n+1)]:n)i})\right)\right) \left(-2J^{w_1}(X_{(\frac{n+1}{2}:n)n})\right) \text{ if } \mbox{n\  odd},
\end{cases}\nonumber \\
&=\begin{cases}
-\frac{1}{2}\left(\prod_{i=1}^{n/2}\left(\int_{-\infty}^{\infty}w_1(x)f_{[p(n+1)]:n}^2(x)dx)\right)\right)\left(\prod_{i=\frac{n}{2}+1}^{n}\left(\int_{-\infty}^{\infty}w_1(x)f_{[q(n+1)]:n}^2(x)dx\right) \right) \text{ if } \mbox{n\  even},\\
-\frac{1}{2}\left(\prod_{i=1}^{\frac{n-1}{2}}\left(\int_{-\infty}^{\infty}w_1(x)f_{[p(n+1)]:n}^2(x)dx)\right)\right)\left(\prod_{i=\frac{n+1}{2}}^{n-1}\left(\int_{-\infty}^{\infty}w_1(x)f_{[q(n+1)]:n}^2(x)dx)\right)\right)\left(\int_{-\infty}^{\infty}w_1(x)f_{\frac{n+1}{2}:n}^2(x)dx)\right) \text{ if } \mbox{n\  odd},
\end{cases}\nonumber \\
&=\begin{cases} \label{PRSS1}
-\frac{Q_{1,n}}{2}\left(E\left(\Lambda_X^{w_1} (B_{2a-1:2n-1})\right)\right)^{n/2}\left(E\left(\Lambda_X^{w_1} (B_{2b-1:2n-1})\right) \right)^{n/2} \text{ if } \mbox{n\  even},
\\
-\frac{Q_{2,n}}{2}\left(E\left(\Lambda_X^{w_1} (B_{2a-1:2n-1})\right)\right)^{\frac{n-1}{2}} \left(E\left(\Lambda_X^{w_1} (B_{2b-1:2n-1})\right)\right)^{\frac{n-1}{2}}\left(E\left(\Lambda_X^{w_1} (B_{n:2n-1})\right)\right) \text{ if } \mbox{n\  odd},
\end{cases}
\end{align}
	where $a=[p(n+1)]$, $p\in (0,\frac{1}{2})\cup (\frac{1}{2},1)$, $b=[q(n+1)]$, and throughout the paper the notation $ [ . ] $ denotes the number rounded to the nearest integer; here
 \begin{align*}
 Q_{1,n}&= n^{n}\prod_{i=1}^{\frac{n}{2}}C_{a,n} \prod_{i=\frac{n}{2}+1}^{n}C_{b,n}, \  
		Q_{2,n}=n^{n}\prod_{i=1}^{\frac{n-1}{2}}C_{a,n} \prod_{i=\frac{n+1}{2}}^{n-1}C_{b,n} \frac{n((n-1)!)^{4}}{((\frac{n-1}{2})!)^{4}(2n-1)!},\\  C_{a,n}&=\frac{\binom{2a-2}{a-1}\binom{2n-2a}{n-a}}{\binom{2n-1}{n-1}}, \   C_{b,n}=\frac{\binom{2b-2}{b-1}\binom{2n-2b}{n-b}}{\binom{2n-1}{n-1}}, 
  \end{align*}
	and $B_{2i-1:2n-1}$ is a beta distributed random variable, with parameters $(2i-1)$ and $(2n-2i+1)$, having pdf
 \[\phi_{2i-1:2n-1}(u)=\frac{(2n-1)!}{(2i-2)!(2n-2i)!} u^{2i-2}(1-u)^{2n-2i}, 0<u<1.\] Equation $(\ref{PRSS1})$ provides an expression in simplified form of the GWE of $\textbf{X}_{PRSS}^{(n)}$.
 Now we provide some examples to illustrate the equation $(\ref{PRSS1})$.
	
	\begin{example}
		Let $V$ be a random variable with power distribution. The pdf and cdf of $V$ are respectively $f(x)= \theta x^{\theta -1}$ 
		and $F(x)= x^{\theta}$ , $0<x<1$ ,  $\theta > 0$. Let $w_1(x)=x^m,\ x>0, \ m>0$, then it follows that 
        \begin{align*}
            \Lambda_V^{w_1}(u)= w_1(F^{-1}(u)) f(F^{-1}(u))=\theta u^{\frac{m + \theta - 1}{\theta}},
        \end{align*}
        for $ w_1(x)=x^m.$ 
		Then we have
		\begin{align*}	
	& J^{w_1}(\textbf{V}_{PRSS}^{(n)})\nonumber \\ 
 &=\begin{cases} 
-\frac{Q_{1,n}}{2}\left(E\left(\Lambda_X^{w_1} (B_{2a-1:2n-1})\right)\right)^{n/2}\left(E\left(\Lambda_X^{w_1} (B_{2b-1:2n-1})\right) \right)^{n/2} \text{ if } \mbox{n\  even},
\\
-\frac{Q_{2,n}}{2}\left(E\left(\Lambda_X^{w_1} (B_{2a-1:2n-1})\right)\right)^{\frac{n-1}{2}} \left(E\left(\Lambda_X^{w_1} (B_{2b-1:2n-1})\right)\right)^{\frac{n-1}{2}}\left(E\left(\Lambda_X^{w_1} (B_{n:2n-1})\right)\right) \text{ if } \mbox{n\  odd},
\end{cases}\nonumber\\ 
 &=\begin{cases}
-\frac{Q_{1,n}}{2}\left(\prod_{i=1}^{n/2} \int_{0}^{1}\left(\Lambda_V^{w_1}(u) \frac{(2n-1)!}{(2a-2)!(2n-2a)!}u^{2a-2}(1-u)^{2n-2a}du\right)\right) \nonumber \\ \ \ \ \ \ \ \ \ \ \ \ \ \ \ \ \ \ \ \ \ \ \ \ \ \ \ \ \ \ \ \ \ \ \ \ \ \ \  \left(\prod_{i=\frac{n}{2}+1}^{n}\int_{0}^{1}\left(\Lambda_V^{w_1}(u) \frac{(2n-1)!}{(2b-2)!(2n-2b)!}u^{2b-2} (1-u)^{2n-2b}du\right) \right) \text{ if } \mbox{n\  even},\\
-\frac{Q_{2,n}}{2}\left(\prod_{i=1}^{\frac{n-1}{2}} \int_{0}^{1}\left(\Lambda_V^{w_1}(u) \frac{(2n-1)!}{(2a-2)!(2n-2a)!}u^{2a-2}(1-u)^{2n-2a}du\right)\right)  \nonumber \\ \ \ \left(\prod_{i=\frac{n+1}{2}}^{n-1}\int_{0}^{1}\left(\Lambda_V^{w_1}(u) \frac{(2n-1)!}{(2b-2)!(2n-2b)!}u^{2b-2}(1-u)^{2n-2b}du\right) \right)\left( \int_{0}^{1} \Lambda_V^{w_1}(u) \frac{(2n-1)!}{((n-1)!)^{2}} u^{(n-1)}(1-u)^{(n-1)} du\right)\text{ if } \mbox{n\  odd}.
\end{cases}\nonumber \\
&=\begin{cases}
-\frac{Q_{1,n}}{2}\left(\prod_{i=1}^{n/2} \int_{0}^{1}\left(\theta u^{\frac{m + \theta - 1}{\theta}} \frac{(2n-1)!}{(2a-2)!(2n-2a)!}u^{2a-2}(1-u)^{2n-2a}du\right)\right) \nonumber \\ \ \ \ \ \ \ \ \ \ \ \ \ \ \ \ \ \ \ \ \ \ \ \ \ \ \ \ \ \ \ \ \ \ \ \ \ \ \  \left(\prod_{i=\frac{n}{2}+1}^{n}\int_{0}^{1}\left(\theta u^{\frac{m + \theta - 1}{\theta}} \frac{(2n-1)!}{(2b-2)!(2n-2b)!}u^{2b-2} (1-u)^{2n-2b}du\right) \right) \text{ if } \mbox{n\  even},\\
-\frac{Q_{2,n}}{2}\left(\prod_{i=1}^{\frac{n-1}{2}} \int_{0}^{1}\left(\theta u^{\frac{m + \theta - 1}{\theta}} \frac{(2n-1)!}{(2a-2)!(2n-2a)!}u^{2a-2}(1-u)^{2n-2a}du\right)\right)  \nonumber \\ \  \left(\prod_{i=\frac{n+1}{2}}^{n-1}\int_{0}^{1}\left(\theta u^{\frac{m + \theta - 1}{\theta}} \frac{(2n-1)!}{(2b-2)!(2n-2b)!}u^{2b-2}(1-u)^{2n-2b}du\right) \right) \left( \int_{0}^{1} \theta u^{\frac{m + \theta - 1}{\theta}} \frac{(2n-1)!}{((n-1)!)^{2}} u^{(n-1)}(1-u)^{(n-1)} du\right)\text{ if } \mbox{n\  odd}.
\end{cases}\nonumber\end{align*}

\begin{align*}
&=\begin{cases}
-\frac{Q_{1,n}}{2} \theta^{n}\left(\prod_{i=1}^{n/2} \int_{0}^{1}\left( u^{\frac{m + \theta - 1}{\theta}} \frac{(2n-1)!}{(2a-2)!(2n-2a)!}u^{2a-2}(1-u)^{2n-2a}du\right)\right)  \nonumber \\ \ \ \ \ \ \ \ \ \ \ \ \ \ \ \ \ \ \ \ \ \ \ \ \ \ \ \ \ \ \ \ \ \ \ \ \ \ \ \left(\prod_{i=\frac{n}{2}+1}^{n}\int_{0}^{1}\left(u^{\frac{m + \theta - 1}{\theta}} \frac{(2n-1)!}{(2b-2)!(2n-2b)!}u^{2b-2} (1-u)^{2n-2b}du\right) \right) \text{ if } \mbox{n\  even},\\
-\frac{Q_{2,n}}{2} \theta^{n}\left(\prod_{i=1}^{\frac{n-1}{2}} \int_{0}^{1}\left( u^{\frac{m + \theta - 1}{\theta}} \frac{(2n-1)!}{(2a-2)!(2n-2a)!}u^{2a-2}(1-u)^{2n-2a}du\right)\right)  \nonumber \\ \  \left(\prod_{i=\frac{n+1}{2}}^{n-1}\int_{0}^{1}\left( u^{\frac{m + \theta - 1}{\theta}} \frac{(2n-1)!}{(2b-2)!(2n-2b)!}u^{2b-2}(1-u)^{2n-2b}du\right) \right)  \left( \int_{0}^{1}  u^{\frac{m + \theta - 1}{\theta}} \frac{(2n-1)!}{((n-1)!)^{2}} u^{(n-1)}(1-u)^{(n-1)} du\right)\text{ if } \mbox{n\  odd}.
\end{cases}\nonumber\\
&=\begin{cases}
-\frac{Q_{1,n}}{2} \theta^{n}\left(\prod_{i=1}^{n/2} \left(  \frac{(2n-1)!\Gamma(\frac{m+2a\theta-1}{\theta})}{(2a-2)!\Gamma(\frac{m+(2n+1)\theta-1}{\theta})}\right)\right)  \left(\prod_{i=\frac{n}{2}+1}^{n}\left( \frac{(2n-1)!\Gamma(\frac{m+2b\theta-1}{\theta})}{(2b-2)!\Gamma(\frac{m+(2n+1)\theta-1}{\theta})} \right) \right) \text{ if } \mbox{n\  even},\\
-\frac{Q_{2,n}}{2} \theta^{n}\left(\prod_{i=1}^{\frac{n-1}{2}} \left( \frac{(2n-1)!\Gamma(\frac{m+2a\theta-1}{\theta})}{(2a-2)!\Gamma(\frac{m+(2n+1)\theta-1}{\theta})}\right)\right) \left(\prod_{i=\frac{n+1}{2}}^{n-1}\left( \frac{(2n-1)!\Gamma(\frac{m+2b\theta-1}{\theta})}{(2b-2)!\Gamma(\frac{m+(2n+1)\theta-1}{\theta})}\right) \right) \left( \frac{(2n-1)!\Gamma(\frac{m+(n+1)\theta-1}{\theta})}{((n-1)!) \Gamma(\frac{m+(2n+1)\theta-1}{\theta})}\right)\text{ if } \mbox{n\  odd}.
\end{cases}
\end{align*} \hfill $\blacksquare$
	\end{example}
 
	\begin{example}
		Let $W$ have an exponential distribution with cdf $F_W(w)=1-e^{-\lambda w}, \ \lambda >0, \ w>0$. Let $w_1(x)=x^m, \ m>0, \ x>0$, then it follows that
		\begin{align*}
	\Lambda_W^{w_1} (u) = w_1(F^{-1}(u)) f(F^{-1}(u))
		= \frac{(-1)^m (1-u) (ln(1-u))^m}{\lambda^{m-1}}, 0<u<1.	
		\end{align*}
		Then we have 
		\begin{align*}
  & J^{w_1}(\textbf{W}_{PRSS}^{(n)})\nonumber \\ 
 &=\begin{cases} 
-\frac{Q_{1,n}}{2}\left(E\left(\Lambda_X^{w_1} (B_{2a-1:2n-1})\right)\right)^{n/2}\left(E\left(\Lambda_X^{w_1} (B_{2b-1:2n-1})\right) \right)^{n/2} \text{ if } \mbox{n\  even},
\\
-\frac{Q_{2,n}}{2}\left(E\left(\Lambda_X^{w_1} (B_{2a-1:2n-1})\right)\right)^{\frac{n-1}{2}} \left(E\left(\Lambda_X^{w_1} (B_{2b-1:2n-1})\right)\right)^{\frac{n-1}{2}}\left(E\left(\Lambda_X^{w_1} (B_{n:2n-1})\right)\right) \text{ if } \mbox{n\  odd},
\end{cases}\nonumber \\
 &=\begin{cases}
-\frac{Q_{1,n}}{2}\left(\prod_{i=1}^{n/2} \int_{0}^{1}\left(\Lambda_V^{w_1}(u) \frac{(2n-1)!}{(2a-2)!(2n-2a)!}u^{2a-2}(1-u)^{2n-2a}du\right)\right)  \nonumber \\ \ \ \ \ \ \ \ \ \ \ \ \ \ \ \ \ \ \ \ \ \ \ \ \ \ \ \ \ \ \ \ \ \ \ \ \ \ \ \left(\prod_{i=\frac{n}{2}+1}^{n}\int_{0}^{1} 
 \left(\Lambda_V^{w_1}(u) \frac{(2n-1)!}{(2b-2)!(2n-2b)!}u^{2b-2} (1-u)^{2n-2b}du\right) \right) \text{ if } \mbox{n\  even},\\
-\frac{Q_{2,n}}{2}\left(\prod_{i=1}^{\frac{n-1}{2}} \int_{0}^{1}\left(\Lambda_V^{w_1}(u) \frac{(2n-1)!}{(2a-2)!(2n-2a)!}u^{2a-2}(1-u)^{2n-2a}du\right)\right)  \nonumber \\ \  \left(\prod_{i=\frac{n+1}{2}}^{n-1}\int_{0}^{1}\left(\Lambda_V^{w_1}(u) \frac{(2n-1)!}{(2b-2)!(2n-2b)!}u^{2b-2}(1-u)^{2n-2b}du\right) \right) \left( \int_{0}^{1} \Lambda_V^{w_1}(u) \frac{(2n-1)!}{((n-1)!)^{2}} u^{(n-1)}(1-u)^{(n-1)} du\right)\text{ if } \mbox{n\  odd}.
\end{cases}\nonumber \\
 &=\begin{cases}
-\frac{Q_{1,n}}{2}\left(\prod_{i=1}^{n/2} \int_{0}^{1}\frac {(-1)^{m}(1-u)(ln(1-u))^{m}}{\lambda^{m-1}} \frac{(2n-1)!}{(2a-2)!(2n-2a)!}u^{2a-2}(1-u)^{2n-2a}du\right)  \nonumber \\ \ \ \ \ \ \ \ \ \ \ \ \ \ \ \ \ \ \ \ \ \ \ \ \ \ \ \ \ \ \ \ \ \ \ \ \ \ \ \left(\prod_{i=\frac{n}{2}+1}^{n}\int_{0}^{1} \frac {(-1)^{m}(1-u)(ln(1-u))^{m}}{\lambda^{m-1}} \frac{(2n-1)!}{(2b-2)!(2n-2b)!}u^{2b-2}(1-u)^{2n-2b}du \right) \text{ if } \mbox{n\  even},\\
-\frac{Q_{2,n}}{2}\left(\prod_{i=1}^{\frac{n-1}{2}} \int_{0}^{1}\frac {(-1)^{m}(1-u)(ln(1-u))^{m}}{\lambda^{m-1}} \frac{(2n-1)!}{(2a-2)!(2n-2a)!}u^{2a-2}(1-u)^{2n-2a}du\right)  \nonumber \\ \ \left(\prod_{i=\frac{n+1}{2}}^{n-1}\int_{0}^{1}\frac {(-1)^{m}(1-u)(ln(1-u))^{m}}{\lambda^{m-1}} \frac{(2n-1)!}{(2b-2)!(2n-2b)!}u^{2b-2}(1-u)^{2n-2b}du\right)  \nonumber \\ \ \ \ \ \ \ \ \ \ \ \ \ \ \ \ \ \ \ \ \ \ \ \ \ \ \ \ \ \ \ \ \ \ \ \ \ \ \ \left( \int_{0}^{1} \frac {(-1)^{m}(1-u)(ln(1-u))^{m}}{\lambda^{m-1}} \frac{(2n-1)!}{((n-1)!)^{2}} u^{(n-1)}(1-u)^{(n-1)} du\right)\text{ if } \mbox{n\  odd}.
\end{cases}\nonumber \\
		\end{align*}
		Taking $u=1-e^{-x}$ in the above equation, we get
		\begin{align*}
			& J^{w_1}(\textbf{W}_{PRSS}^{(n)})\\ 
			&=\begin{cases}
-\frac{Q_{1,n}}{2}\left(\prod_{i=1}^{n/2} \frac{(-1)^{m}}{\lambda^{m-1}} \int_{0}^{\infty} e^{-x} (ln( e^{-x}))^{m}\frac{(2n-1)!}{(2a-2)!(2n-2a)!}(1-e^{-x})^{2a-2}(e^{-x})^{2n-2a} e^{-x}dx\right) \nonumber \\ \  \left(\prod_{i=\frac{n}{2}+1}^{n} \frac{(-1)^{m}}{\lambda^{m-1}} \int_{0}^{\infty}e^{-x}(ln(e^{-x}))^{m} \frac{(2n-1)!}{(2b-2)!(2n-2b)!}(1-e^{-x})^{2b-2}(e^{-x})^{2n-2b} e^{-x}dx \right) \text{ if } \mbox{n\  even},\\
-\frac{Q_{2,n}}{2}\left(\prod_{i=1}^{\frac{n-1}{2}} \frac{(-1)^{m}}{\lambda^{m-1}} \int_{0}^{\infty} e^{-x} (ln( e^{-x}))^{m}\frac{(2n-1)!}{(2a-2)!(2n-2a)!}(1-e^{-x})^{2a-2}(e^{-x})^{2n-2a} e^{-x}dx\right) \nonumber \\ \  \left(\prod_{i=\frac{n+1}{2}}^{n-1}\frac{(-1)^{m}}{\lambda^{m-1}} \int_{0}^{\infty}e^{-x}(ln(e^{-x}))^{m} \frac{(2n-1)!}{(2b-2)!(2n-2b)!}(1-e^{-x})^{2b-2}(e^{-x})^{2n-2b} e^{-x}dx  \right) \nonumber \\ \ \  \left( \int_{0}^{\infty} \frac {(-1)^{m} e^{-x}(ln(e^{-x}))^{m}}{\lambda^{m-1}} \frac{(2n-1)!}{((n-1)!)^{2}} (1-e^{-x})^{(n-1)}(e^{-x})^{(n-1)} e^{-x} dx\right)\text{ if } \mbox{n\  odd}.
\end{cases}\nonumber \\
&=\begin{cases}
-\frac{Q_{1,n}}{2} \frac{1}{\lambda^{n(m-1)}}\left(\prod_{i=1}^{n/2} \int_{0}^{\infty} x^{m} \frac{(2n-1)!}{(2a-2)!(2n-2a)!}(1-e^{-x})^{2a-2}(e^{-x})^{2n-2a+2} dx\right)  \nonumber \\ \ \ \ \ \ \ \ \ \ \ \ \ \ \ \ \ \ \ \ \ \ \ \ \ \ \ \ \ \ \ \ \ \ \ \ \ \ \ \left(\prod_{i=\frac{n}{2}+1}^{n} \int_{0}^{\infty} x^{m} \frac{(2n-1)!}{(2b-2)!(2n-2b)!}(1-e^{-x})^{2b-2}(e^{-x})^{2n-2b+2} dx \right) \text{ if } \mbox{n\  even},\\
-\frac{Q_{2,n}}{2} \frac{1}{\lambda^{n(m-1)}}\left(\prod_{i=1}^{\frac{n-1}{2}} \int_{0}^{\infty} x^{m} \frac{(2n-1)!}{(2a-2)!(2n-2a)!}(1-e^{-x})^{2a-2}(e^{-x})^{2n-2a+2} dx\right)  \nonumber \\ \ \left(\prod_{i=\frac{n+1}{2}}^{n-1}\int_{0}^{\infty} x^{m} \frac{(2n-1)!}{(2b-2)!(2n-2b)!}(1-e^{-x})^{2b-2}(e^{-x})^{2n-2b+2} dx \right) \left( \int_{0}^{\infty} \frac{1}{2} x^{m} \frac {(2n)!}{(n-1)!(n!)} (1-e^{-x})^{n-1} (e^{-x})^{n+1}dx\right)\text{ if } \mbox{n\  odd}.
\end{cases}\nonumber \\
&=\begin{cases}
-\frac{Q_{1,n}}{2} \frac{1}{\lambda^{n(m-1)}}\left(\prod_{i=1}^{n/2} \frac{(2n-2a+1)}{2n}\int_{0}^{\infty} x^{m} \frac{(2n)!}{(2a-2)!(2n-2a+1)!} (1-e^{-x})^{2a-2}(e^{-x})^{2n-2a+2} dx\right)  \nonumber \\ \ \ \ \ \ \ \ \ \ \ \ \ \ \ \ \ \ \ \ \ \ \ \ \ \ \ \ \ \ \ \ \ \ \ \ \ \ \ \left(\prod_{i=\frac{n}{2}+1}^{n} \frac{(2n-2b+1)}{2n}\int_{0}^{\infty} x^{m} \frac{(2n)!}{(2b-2)!(2n-2b+1)!} (1-e^{-x})^{2b-2}(e^{-x})^{2n-2b+2} dx \right) \text{ if } \mbox{n\  even},\\
-\frac{Q_{2,n}}{2} \frac{1}{\lambda^{n(m-1)}}\left(\prod_{i=1}^{\frac{n-1}{2}}\frac{(2n-2a+1)}{2n}\int_{0}^{\infty} x^{m} \frac{(2n)!}{(2a-2)!(2n-2a+1)!} (1-e^{-x})^{2a-2}(e^{-x})^{2n-2a+2} dx\right)   \nonumber \\ \  \left(\prod_{i=\frac{n+1}{2}}^{n-1}\frac{(2n-2b+1)}{2n}\int_{0}^{\infty} x^{m} \frac{(2n)!}{(2b-2)!(2n-2b+1)!} (1-e^{-x})^{2b-2}(e^{-x})^{2n-2b+2} dx\right)   \nonumber \\ \ \ \ \ \ \ \ \ \ \ \ \ \ \ \ \ \ \ \ \ \ \ \ \ \ \ \ \ \ \ \ \ \ \ \ \ \ \ \left( \int_{0}^{\infty} \frac{1}{2} x^{m} \frac {(2n)!}{(n-1)!(n!)} (1-e^{-x})^{n-1} (e^{-x})^{n+1}dx\right)\text{ if } \mbox{n\  odd}.
\end{cases}\nonumber \\
&=\begin{cases}
-\frac{Q_{1,n}(2n-1)!!}{2^{n+1}n^{n}} \frac{1}{\lambda^{n(m-1)}}\left(\prod_{i=1}^{n/2} E(W_{2a-1:2n}^{m})\right)\left(\prod_{i=\frac{n}{2}+1}^{n} E(W_{2b-1:2n}^{m}) \right) \text{ if } \mbox{n\  even},\\
-\frac{Q_{2,n}(2n-1)!!}{2^{n+1}n^{n-1}} \frac{1}{\lambda^{n(m-1)}}\left(\prod_{i=1}^{\frac{n-1}{2}} E(W_{2a-1:2n}^{m})\right) \left(\prod_{i=\frac{n+1}{2}}^{n-1}E(W_{2b-1:2n}^{m}) \right) \left( E(W_{n:2n}^{m})\right)\text{ if } \mbox{n\  odd}.
\end{cases}\nonumber \\
		\end{align*}
		where $W_{2i-1:2n}$ is the $(2i-1)$-th order statistics of a sample of size $2n$ from exponential distribution having pdf given by 
		$$\psi_{2i-1:2n}=\frac{(2n)!}{(2i-2)!(2n-2i+1)!}(1-e^{-x})^{2i-2}(e^{-x})^{2n-2i+2},\  x\ge 0,$$ 
		and  $(2n-1)!!= (2n-2a+1)^{\frac{n}{2}}(2n-2b+1)^{\frac{n}{2}} $ if $n$ is even and $(2n-1)!!= (2n-2a+1)^{\frac{n-1}{2}}(2n-2b+1)^{\frac{n-1}{2}} $if $n$ is odd.\hfill $\blacksquare$
	\end{example}
	
	\begin{example}
		Let $U$ be a Pareto random variable with cdf $F(x)=1-x^{-\alpha},\  \alpha >0, \ x>1 $. Let $w_1(x)=x^m,\ m>0, \ x>0 $, then we get
		\begin{align*}
			\Lambda_U^{w_1} (u) 
			&= w_1(F^{-1}(u)) f(F^{-1}(u))\\
			&=\alpha (1-u)^{\frac{\alpha - m+1}{\alpha}}.		
		\end{align*}
		The weighted extropy of $\textbf{U}_{PRSS}^{(n)}$ is 
		\begin{align*}
			& J^{w_1}(\textbf{U}_{PRSS}^{(n)}) \\
			 &=\begin{cases} 
-\frac{Q_{1,n}}{2}\left(E\left(\Lambda_X^{w_1} (B_{2a-1:2n-1})\right)\right)^{n/2}\left(E\left(\Lambda_X^{w_1} (B_{2b-1:2n-1})\right) \right)^{n/2} \text{ if } \mbox{n\  even},
\\
-\frac{Q_{2,n}}{2}\left(E\left(\Lambda_X^{w_1} (B_{2a-1:2n-1})\right)\right)^{\frac{n-1}{2}} \left(E\left(\Lambda_X^{w_1} (B_{2b-1:2n-1})\right)\right)^{\frac{n-1}{2}}\left(E\left(\Lambda_X^{w_1} (B_{n:2n-1})\right)\right) \text{ if } \mbox{n\  odd},
\end{cases}\nonumber \\
 &=\begin{cases}
-\frac{Q_{1,n}}{2}\left(\prod_{i=1}^{n/2} \int_{0}^{1}\left(\Lambda_V^{w_1}(u) \frac{(2n-1)!}{(2a-2)!(2n-2a)!}u^{2a-2}(1-u)^{2n-2a}du\right)\right)  \nonumber \\ \ \ \ \ \ \ \ \ \ \ \ \ \ \ \ \ \ \ \ \ \ \ \ \ \ \ \ \ \ \ \ \ \ \ \ \ \ \  \left(\prod_{i=\frac{n}{2}+1}^{n}\int_{0}^{1}\left(\Lambda_V^{w_1}(u) \frac{(2n-1)!}{(2b-2)!(2n-2b)!}u^{2b-2} (1-u)^{2n-2b}du\right) \right) \text{ if } \mbox{n\  even},\\
-\frac{Q_{2,n}}{2}\left(\prod_{i=1}^{\frac{n-1}{2}} \int_{0}^{1}\left(\Lambda_V^{w_1}(u) \frac{(2n-1)!}{(2a-2)!(2n-2a)!}u^{2a-2}(1-u)^{2n-2a}du\right)\right)  \nonumber \\ \ \left(\prod_{i=\frac{n+1}{2}}^{n-1}\int_{0}^{1}\left(\Lambda_V^{w_1}(u) \frac{(2n-1)!}{(2b-2)!(2n-2b)!}u^{2b-2}(1-u)^{2n-2b}du\right) \right)  \left( \int_{0}^{1} \Lambda_V^{w_1}(u) \frac{(2n-1)!}{((n-1)!)^{2}} u^{(n-1)}(1-u)^{(n-1)} du\right)\text{ if } \mbox{n\  odd}.
\end{cases}\nonumber\\
 &=\begin{cases}
-\frac{Q_{1,n}}{2}\left(\prod_{i=1}^{n/2} \int_{0}^{1}\left(\alpha (1-u)^{\frac{\alpha - m+1}{\alpha}} \frac{(2n-1)!}{(2a-2)!(2n-2a)!}u^{2a-2}(1-u)^{2n-2a}du\right)\right)  \nonumber \\ \ \ \ \ \ \ \ \ \ \ \ \ \ \ \ \ \ \ \ \ \ \ \ \ \ \ \ \ \ \ \ \ \ \ \ \ \ \ \left(\prod_{i=\frac{n}{2}+1}^{n}\int_{0}^{1}\left(\alpha (1-u)^{\frac{\alpha - m+1}{\alpha}} \frac{(2n-1)!}{(2b-2)!(2n-2b)!}u^{2b-2} (1-u)^{2n-2b}du\right) \right) \text{ if } \mbox{n\  even},\\
-\frac{Q_{2,n}}{2}\left(\prod_{i=1}^{\frac{n-1}{2}} \int_{0}^{1} \left(\alpha (1-u)^{\frac{\alpha - m+1}{\alpha}} \frac{(2n-1)!}{(2a-2)!(2n-2a)!}u^{2a-2}(1-u)^{2n-2a}du\right)\right)  \nonumber \\ \  \left(\prod_{i=\frac{n+1}{2}}^{n-1}\int_{0}^{1}\left(\alpha (1-u)^{\frac{\alpha - m+1}{\alpha}} \frac{(2n-1)!}{(2b-2)!(2n-2b)!}u^{2b-2}(1-u)^{2n-2b}du\right) \right)  \nonumber \\ \ \ \ \ \ \ \ \ \ \ \ \ \ \ \ \ \ \ \ \ \ \ \ \ \ \ \ \ \ \ \ \ \ \ \ \ \ \ \left( \int_{0}^{1} \alpha (1-u)^{\frac{\alpha - m+1}{\alpha}} \frac{(2n-1)!}{((n-1)!)^{2}} u^{(n-1)}(1-u)^{(n-1)} du\right)\text{ if } \mbox{n\  odd}.
\end{cases}\nonumber \\
&=\begin{cases}
-\frac{Q_{1,n}}{2} \alpha^{n}\left(\prod_{i=1}^{n/2} \int_{0}^{1}\left(  \frac{(2n-1)!}{(2a-2)!(2n-2a)!}u^{2a-2}(1-u)^{2n-2a+\frac{\alpha - m+1}{\alpha}}du\right)\right)  \nonumber \\ \ \ \ \ \ \ \ \ \ \ \ \ \ \ \ \ \ \ \ \ \ \ \ \ \ \ \ \ \ \ \ \ \ \ \ \ \ \  \left(\prod_{i=\frac{n}{2}+1}^{n}\int_{0}^{1}\left( \frac{(2n-1)!}{(2b-2)!(2n-2b)!}u^{2b-2} (1-u)^{2n-2b+\frac{\alpha - m+1}{\alpha}}du\right) \right) \text{ if } \mbox{n\  even},\\
-\frac{Q_{2,n}}{2}\alpha^{n}\left(\prod_{i=1}^{n/2} \int_{0}^{1}\left(  \frac{(2n-1)!}{(2a-2)!(2n-2a)!}u^{2a-2}(1-u)^{2n-2a+\frac{\alpha - m+1}{\alpha}}du\right)\right)  \nonumber \\ \ \left(\prod_{i=\frac{n}{2}+1}^{n}\int_{0}^{1}\left( \frac{(2n-1)!}{(2b-2)!(2n-2b)!}u^{2b-2} (1-u)^{2n-2b+\frac{\alpha - m+1}{\alpha}}du\right) \right) \left( \int_{0}^{1} \frac{(2n-1)!}{((n-1)!)^{2}} u^{(n-1)}(1-u)^{(n-1+\frac{\alpha - m+1}{\alpha})} du\right)\text{ if } \mbox{n\  odd}.
\end{cases}\nonumber \\
&=\begin{cases}
-\frac{Q_{1,n}}{2} \alpha^{n}\left(\prod_{i=1}^{n/2} \left(  \frac{(2n-1)!}{(2n-2a)!}\frac{\Gamma(\frac{2n\alpha-2a\alpha+2\alpha-m+1}{\alpha})}{\Gamma(\frac{\alpha+2n\alpha-m+1}{\alpha})}\right)\right)  \left(\prod_{i=\frac{n}{2}+1}^{n}\left( \frac{(2n-1)!}{(2n-2b)!}\frac{\Gamma(\frac{2n\alpha-2b\alpha+2\alpha-m+1}{\alpha})}{\Gamma(\frac{\alpha+2n\alpha-m+1}{\alpha})}\right) \right) \text{ if } \mbox{n\  even},\\
-\frac{Q_{2,n}}{2}\alpha^{n}\left(\prod_{i=1}^{n/2} \left(  \frac{(2n-1)!}{(2n-2a)!}\frac{\Gamma(\frac{2n\alpha-2a\alpha+2\alpha-m+1}{\alpha})}{\Gamma(\frac{\alpha+2n\alpha-m+1}{\alpha})}\right)\right)  \left(\prod_{i=\frac{n}{2}+1}^{n}\int_{0}^{1}\left( \frac{(2n-1)!}{(2n-2b)!}\frac{\Gamma(\frac{2n\alpha-2b\alpha+2\alpha-m+1}{\alpha})}{\Gamma(\frac{\alpha+2n\alpha-m+1}{\alpha})}\right) \right) \nonumber \\ \ \ \ \ \ \ \ \ \ \ \ \ \ \ \ \ \ \ \ \ \ \ \ \ \ \ \ \ \ \ \ \ \ \ \ \ \ \ \left(  \frac{(2n-1)!}{((n-1)!)} \frac{\Gamma(\frac{\alpha-m+1}{\alpha}+n)}{\Gamma(\frac{(2n+1)\alpha-m+1}{\alpha})}\right)\text{ if } \mbox{n\  odd}.
\end{cases}\nonumber \\
		\end{align*}\hfill $\blacksquare$

	\end{example}

		The following result gives the conditions under which the GWE will increase (decrease).	
		
		\begin{theorem}\label{new1}
			Let $X$ be a non-negative absolutely continuous random variable with pdf f and cdf F. Assume $\eta(x)$ is an increasing function and  $\frac{w_1(\eta(x))}{\eta^\prime (x)} \leq (\geq) w_1(x)$ and $\eta(0)=0$. If $V=\eta(X)$, then $J^{w_1}(\textbf{X}_{PRSS}^{(n)})\leq (\geq) J^{w_1}(\textbf{V}_{PRSS}^{(n)})$.

		\end{theorem}
		\begin{proof}
			The proof is on similar lines as the proof of theorem 3.1 of Gupta and Chaudhary (2023).
		\end{proof}
		
		\begin{example}
			Let $Z$ have an exponential distribution with cdf $F_Z(z)=1-e^{-\lambda z}, \ \lambda >0, \ z>0$. Let $w_1(x)=x^{2}>0$. Consider $\eta(x)=e^x-1,\ x \geq 0$. Then $\eta (Z)$
			is Pareto distribution (see  Qiu and Raqab (2022) Example 2.7) with survival function $\bar F_{\eta (Z)}(x)=1/(1+x)^\lambda$, $x\geq 0$. Note that 
			\begin{align*}
				\frac{w_1(\eta(x))}{\eta^\prime (x)}=\frac{(e^x-1)^2}{e^x}= e^{x}+e^{-x}-2\geq x^{2}=w_1(x).
			\end{align*}
			Hence using Theorem \ref{new1}, $J^{w_1}(\textbf{Z}_{PRSS}^{(n)})\geq J^{w_1}(\eta(\textbf{Z})_{PRSS}^{(n)})$.
		\end{example}

		\noindent A lower bound for the overall weighted extropy of PRSS data is established, relying on the weighted extropy of the SRS data, as demonstrated in the subsequent result.
		
		\begin{theorem}
			Let $X$ be an absolutely continuous random variable with pdf f and cdf F. For $n$ even, 
		\end{theorem}
		\begin{align*}
			\frac{ J^{w_1}(\textbf{X}_{PRSS}^{(n)})}{ J^{w_1}(\textbf{X}_{SRS}^{(n)})}\leq \frac{n^{2n}}{(n-1)^{2n(n-1)}}\left(\prod_{i=1}^{\frac{n}{2}}\left( \binom{n-1}{a-1}^{2} (a-1)^{2a-2}(n-a)^{2n-2a}\right)\right) \nonumber \\ \ \ \ \ \ \ \ \ \ \ \ \ \ \ \ \ \ \ \ \ \ \ \ \ \ \ \ \ \ \ \ \ \ \ \ \ \ \ \left(\prod_{i=\frac{n}{2}+1}^{n}\left (\binom{n-1}{b-1}^{2} (b-1)^{2b-2}(n-b)^{2n-2b}\right)\right),
		\end{align*}
           and for $n$ odd, we have 
           \begin{align*}
			\frac{ J^{w_1}(\textbf{X}_{PRSS}^{(n)})}{ J^{w_1}(\textbf{X}_{SRS}^{(n)})}\leq \frac{n^{2n-2}}{(n-1)^{2(n-1)^{2}}} \frac{(n!)^{2}}{((\frac{n-1}{2})!)^{4}} \left(\prod_{i=1}^{\frac{n-1}{2}} \left( \binom{n-1}{a-1}^{2} (a-1)^{2a-2}(n-a)^{2n-2a}\right)\right)  \nonumber \\ \ \ \ \ \ \ \ \ \ \ \ \ \ \ \ \ \ \ \ \ \ \ \ \ \ \ \ \left(\prod_{i=\frac{n+1}{2}}^{n-1}\left( \binom{n-1}{b-1}^{2} (b-1)^{2b-2}(n-b)^{2n-2b}\right)\right).
		\end{align*}
		\begin{proof}For all $1 < i < n$, consider the function $f(u)= u^{2a-2}(1-u)^{2n-2a}$ and $ h(u) = u^{2b-2}(1-u)^{2n-2b}$, $0 \leq u \leq 1$. It is obvious that $f(u)$ and $h(u)$ obtains their maximum point at $u=\frac{a-1}{n-1}$ and $ u=\frac{b-1}{n-1}$ respectively. That is, $\max \limits_{0\leq u \leq 1}f(u)=\frac{(a-1)^{2a-2}(n-a)^{2n-2a}}{(n-1)^{2n-2}}$ and $\max \limits_{0\leq u \leq 1}h(u)=\frac{(b-1)^{2b-2}(n-b)^{2n-2b}}{(n-1)^{2n-2}}$. For $n$ even, 
           \begin{align*}
                J^{w_1}(\textbf{X}_{PRSS}^{(n)}) &=\frac{-1}{2}\left( \prod_{i=1}^{\frac{n}{2}}\int_{0}^{1} n^{2} \binom{n-1}{a-1}^{2} \Lambda_{X}^{w_1}(u) u^{2a-2}(1-u)^{2n-2a} du \right) \nonumber \\ & \ \ \ \ \ \ \ \ \ \ \ \ \ \ \ \ \ \ \ \ \ \ \ \ \ \ \ \ \ \ \ \ \ \ \ \ \ \ \left(\prod_{i=\frac{n}{2}+1}^{n}\int_{0}^{1} n^{2} \binom{n-1}{b-1}^{2} \Lambda_{X}^{w_1}(u) u^{2b-2}(1-u)^{2n-2b}du\right)\\
              & = \frac{-1}{2}\left( \prod_{i=1}^{\frac{n}{2}}\int_{0}^{1} n^{2} \binom{n-1}{a-1}^{2} \Lambda_{X}^{w_1}(u) f(u) du\right) \left( \prod_{i=\frac{n}{2}+1}^{n-1}\int_{0}^{1} n^{2} \binom{n-1}{b-1}^{2} \Lambda_{X}^{w_1}(u) h(u)du \right ) \\
              & \geq \frac{-1}{2}\left( \prod_{i=1}^{\frac{n}{2}}\int_{0}^{1} n^{2} \binom{n-1}{a-1}^{2} \frac{(a-1)^{2a-2}(n-a)^{2n-2a}}{(n-1)^{2n-2}}\Lambda_{X}^{w_1}(u)du\right)  \\ & \ \ \ \ \ \ \ \ \ \ \ \left( \prod_{i=\frac{n}{2}+1}^{n}\int_{0}^{1} n^{2} \binom{n-1}{b-1}^{2} \frac{(b-1)^{2b-2}(n-b)^{2n-2b}}{(n-1)^{2n-2}}\Lambda_{X}^{w_1}(u)du \right ) \\
              & =\frac{-1}{2} n^{2n} \left( \prod_{i=1}^{\frac{n}{2}}  \binom{n-1}{a-1}^{2} \frac{(a-1)^{2a-2}(n-a)^{2n-2a}}{(n-1)^{2n-2}}(-2J^{w_{1}}(X))\right)   \\ & \ \ \ \ \ \ \ \ \ \ \ \left(\prod_{i=\frac{n}{2}+1}^{n}\int_{0}^{1} \binom{n-1}{b-1}^{2} \frac{(b-1)^{2b-2}(n-b)^{2n-2b}}{(n-1)^{2n-2}}(-2J^{w_{1}}(X)) \right ) \\
              & =\frac{-1}{2} (-2J^{w_{1}}(X))^{n}\frac{n^{2n}}{(n-1)^{2 n (n-1)}}\left(\prod_{i=1}^{\frac{n}{2}}\left( \binom{n-1}{a-1}^{2} (a-1)^{2a-2}(n-a)^{2n-2a}\right)\right)  \\ & \ \ \ \ \ \ \ \ \ \ \  \left(\prod_{i=\frac{n}{2}+1}^{n}\left( \binom{n-1}{b-1}^{2} (b-1)^{2b-2}(n-b)^{2n-2b}\right)\right)\\
              & = J^{w_1}(\textbf{X}_{SRS}^{(n)}) \frac{n^{2n}}{(n-1)^{2n(n-1)}}\left(\prod_{i=1}^{\frac{n}{2}}\left( \binom{n-1}{a-1}^{2} (a-1)^{2a-2}(n-a)^{2n-2a}\right)\right)   \\ & \ \ \ \ \ \ \ \ \ \ \ \left(\prod_{i=\frac{n}{2}+1}^{n}\left (\binom{n-1}{b-1}^{2} (b-1)^{2b-2}(n-b)^{2n-2b}\right)\right)
           \end{align*}
           This completes the proof for $n$ even.\\
           For $n$ odd, we have \\
           \begin{align*}
               & J^{w_1}(\textbf{X}_{PRSS}^{(n)}) = \frac{-1}{2} \left(\prod_{i=1}^{\frac{n-1}{2}}\int_{0}^{1} n^{2} \binom{n-1}{a-1}^{2} \Lambda_{X}^{w_1}(u) u^{2a-2}(1-u)^{2n-2a} du \right) \nonumber \\ & \ \ \ \ \ \ \ \ \  \left( \prod_{i=\frac{n+1}{2}}^{n-1}\int_{0}^{1} n^{2} \binom{n-1}{b-1}^{2} \Lambda_{X}^{w_1}(u) u^{2b-2}(1-u)^{2n-2b}du\right)   \left(\frac{(n!)^{2}}{((\frac{n-1}{2})!)^{4}}\int_{0}^{1}\Lambda_{X}^{w_1}(u) u^{n-1}(1-u)^{n-1} du\right)\\
                     & =  \frac{-1}{2}\left( \prod_{i=1}^{\frac{n-1}{2}}\int_{0}^{1} n^{2} \binom{n-1}{a-1}^{2} \Lambda_{X}^{w_1}(u) f(u) du\right) \left( \prod_{i=\frac{n+1}{2}}^{n-1}\int_{0}^{1} n^{2} \binom{n-1}{b-1}^{2} \Lambda_{X}^{w_1}(u) h(u)du \right ) \nonumber \\ & \ \ \ \ \ \ \ \ \ \ \ \ \ \ \ \left(\frac{(n!)^{2}}{((\frac{n-1}{2})!)^{4}}\int_{0}^{1}\Lambda_{X}^{w_1}(u) u^{n-1}(1-u)^{n-1} du\right)\\
              &  \geq \frac{-1}{2}\left( \prod_{i=1}^{\frac{n-1}{2}}\int_{0}^{1} n^{2} \binom{n-1}{a-1}^{2} \frac{(a-1)^{2a-2}(n-a)^{2n-2a}}{(n-1)^{2n-2}}\Lambda_{X}^{w_1}(u)du\right) \\ & \ \ \ \ \ \ \ \ \ \ \ \ \ \ \ \left( \prod_{i=\frac{n+1}{2}}^{n-1}\int_{0}^{1} n^{2} \binom{n-1}{b-1}^{2} \frac{(b-1)^{2b-2}(n-b)^{2n-2b}}{(n-1)^{2n-2}}\Lambda_{X}^{w_1}(u)du \right )  \left(\frac{(n!)^{2}}{((\frac{n-1}{2})!)^{4}}\int_{0}^{1}\Lambda_{X}^{w_1}(u)  du\right)\\
              & =\frac{-1}{2} n^{2n-2}\left( \prod_{i=1}^{\frac{n-1}{2}} \binom{n-1}{a-1}^{2} \frac{(a-1)^{2a-2}(n-a)^{2n-2a}}{(n-1)^{2n-2}}(-2J^{w_{1}}(X))\right) \\ & \ \ \ \ \ \ \ \ \ \ \ \ \ \ \ \left( \prod_{i=\frac{n+1}{2}}^{n-1} \binom{n-1}{b-1}^{2} \frac{(b-1)^{2b-2}(n-b)^{2n-2b}}{(n-1)^{2n-2}}(-2J^{w_{1}}(X)) \right )   \left(\frac{(n!)^{2}}{((\frac{n-1}{2})!)^{4}} (-2J^{w_{1}}(X))\right) \\
              & =\frac{-1}{2} (-2J^{w_{1}}(X))^{n}\frac{n^{2n-2}}{(n-1)^{2(n-1)^{2}}} \frac{(n!)^{2}}{((\frac{n-1}{2})!)^{4}} \left(\prod_{i=1}^{\frac{n-1}{2}}  \left( \binom{n-1}{a-1}^{2} (a-1)^{2a-2}(n-a)^{2n-2a}\right)\right) \\ & \ \ \ \ \ \ \ \ \ \ \ \ \ \ \  \left(\prod_{i=\frac{n+1}{2}}^{n-1}\left( \binom{n-1}{b-1}^{2} (b-1)^{2b-2}(n-b)^{2n-2b}\right)\right)\\
              & = J^{w_1}(\textbf{X}_{SRS}^{(n)}) \frac{n^{2n-2}}{(n-1)^{2(n-1)^{2}}} \frac{(n!)^{2}}{((\frac{n-1}{2})!)^{4}} \left(\prod_{i=1}^{\frac{n-1}{2}} \left( \binom{n-1}{a-1}^{2} (a-1)^{2a-2}(n-a)^{2n-2a}\right)\right)  \\ & \ \ \ \ \ \ \ \ \ \ \ \ \ \ \ \left(\prod_{i=\frac{n+1}{2}}^{n-1}\left( \binom{n-1}{b-1}^{2} (b-1)^{2b-2}(n-b)^{2n-2b}\right)\right)
           \end{align*}
           This completes the proof for $n$ odd.\hfill $\blacksquare$
		\end{proof}

		\section {Characterization results}

		 The lemma introduced by Fashandi and Ahmadi (2012) is pivotal as a foundational tool for developing various characterizations of symmetric distributions.
         \begin{lemma}
             Let $X$ be a random variable with cdf $F$, pdf $f$, and mean $\mu$.Then $f(\mu+x)=f(\mu-x)$ for all $x \geq 0$ if and only if
             $f(F^{-1}(u))=f(F^{-1}(1-u))$. 
         \end{lemma}
		\begin{theorem}
			Let $X$ be an absolutely continuous random variable with pdf f and cdf F; and assume $w_1(-x)=-w_1(x)$. Then $X$ is a symmetric distributed random variable with mean 0 if and only if	$J^{w_1}(\textbf{X}_{PRSS}^{(n)})=0$  for all odd $n\geq 1$. 
		\end{theorem}
		\begin{proof}
			For sufficiency, suppose $f(x)=f(-x)$ for all $x\geq 0$. Also since $F^{-1}(u)=- F^{-1}(1-u)$,  $f(F^{-1}(u))=f(F^{-1}(1-u))$ for all $0<u<1$ and $w_1(-x)=-w_1(x)$, which implies that
			\begin{equation*}
				\Lambda_X^{w_1} (u)=w_1(F^{-1}(u))f(F^{-1}(u))=-w_1(F^{-1}(1-u)) f(F^{-1}(1-u))=-\Lambda_X^{w_1} (1-u)
			\end{equation*}
			For $n$ odd, we have 
           \begin{align*}
              & J^{w_1}(\textbf{X}_{PRSS}^{(n)})\\
               &=-\frac{Q_{2,n}}{2} \left( \prod_{i=1}^{\frac{n-1}{2}}\int_{0}^{1} \Lambda_{X}^{w}(u)\phi_{2a-1:2n-1}(u)du \right) \left( \prod_{i=\frac{n+1}{2}}^{n-1}\int_{0}^{1}  \Lambda_{X}^{w}(u) \phi_{2b-1:2n-1}(u)du \right) \\ & \ \ \ \ \ \ \ \ \ \ \ \ \ \ \ \left( \int_{0}^{1} \Lambda_{X}^{w}(u) \phi_{n:2n-1}(u)du\right)\\
              & = -\frac{Q_{2,n}}{2} \left( \prod_{i=1}^{\frac{n-1}{2}} -\int_{0}^{1} \Lambda_{X}^{w}(1-u)(2n-1) \binom{2n-2}{2a-2} u^{2a-2}(1-u)^{2n-2a}du \right)  \\ & \ \ \ \ \ \ \ \ \ \ \ \left(\prod_{i=\frac{n+1}{2}}^{n-1} -\int_{0}^{1}  \Lambda_{X}^{w}(1-u) (2n-1)\binom{2n-2}{2b-2} u^{2b-2}(1-u)^{2n-2b}du \right) \\ & \ \ \ \ \ \ \ \ \ \ \left( -\int_{0}^{1} \Lambda_{X}^{w}(1-u) \frac{(2n-1)!}{((n-1)!)^{2}} u^{n-1} (1-u)^{n-1}du \right)\\
              & =(-1)^{n+1}\frac{Q_{2,n}}{2} \left( \prod_{i=1}^{\frac{n-1}{2}} \int_{0}^{1} \Lambda_{X}^{w}(1-u)(2n-1) \binom{2n-2}{2a-2} u^{2a-2}(1-u)^{2n-2a}du \right)  \\ & \ \ \ \ \ \ \ \ \ \ \ \left(\prod_{i=\frac{n+1}{2}}^{n-1} \int_{0}^{1}  \Lambda_{X}^{w}(1-u) (2n-1)\binom{2n-2}{2b-2} u^{2b-2}(1-u)^{2n-2b}du \right) \\ & \ \ \ \ \ \ \ \ \ \ \left( \int_{0}^{1} \Lambda_{X}^{w}(1-u) \frac{(2n-1)!}{((n-1)!)^{2}} u^{n-1} (1-u)^{n-1}du\right)\end{align*}

\begin{align*}
              & = \frac{Q_{2,n}}{2} \left( \prod_{i=1}^{\frac{n-1}{2}} \int_{0}^{1} \Lambda_{X}^{w}(v)(2n-1) \binom{2n-2}{2a-2} (1-v)^{2a-2}(v)^{2n-2a}dv \right)  \\ & \ \ \ \ \ \ \ \ \ \ \ \left(\prod_{i=\frac{n+1}{2}}^{n-1} \int_{0}^{1}  \Lambda_{X}^{w}(v) (2n-1)\binom{2n-2}{2b-2} (1-v)^{2b-2}(v)^{2n-2b}dv \right) \\ & \ \ \ \ \ \ \ \ \ \ \left( \int_{0}^{1} \Lambda_{X}^{w}(v) \frac{(2n-1)!}{((n-1)!)^{2}} v^{n-1} (1-v)^{n-1}dv\right)\\
               & = \frac{Q_{2,n}}{2} \left( \prod_{i=1}^{\frac{n-1}{2}} \int_{0}^{1} \Lambda_{X}^{w}(v)(2n-1) \binom{2n-2}{2n-2b} v^{2b-2}(1-v)^{2n-2b}dv \right)  \\ & \ \ \ \ \ \ \ \ \ \ \ \left(\prod_{i=\frac{n+1}{2}}^{n-1} \int_{0}^{1}  \Lambda_{X}^{w}(v) (2n-1)\binom{2n-2}{2n-2a} v^{2a-2}(1-v)^{2n-2a}dv \right) \\ & \ \ \ \ \ \ \ \ \ \ \left( \int_{0}^{1} \Lambda_{X}^{w}(v) \frac{(2n-1)!}{((n-1)!)^{2}} v^{n-1} (1-v)^{n-1}dv\right)\\
                &=\frac{Q_{2,n}}{2} \left( \prod_{i=1}^{\frac{n-1}{2}}\int_{0}^{1} \Lambda_{X}^{w}(v)\phi_{2b-1:2n-1}(v)dv\right) \left( \prod_{i=\frac{n+1}{2}}^{n-1}\int_{0}^{1}  \Lambda_{X}^{w}(v) \phi_{2a-1:2n-1}(v)dv \right) \\ & \ \ \ \ \ \ \ \ \ \ \ \ \ \ \ \left(\int_{0}^{1} \Lambda_{X}^{w}(v) \phi_{n:2n-1}(v)dv\right)\\
              & = -  J^{w_1}(\textbf{X}_{PRSS}^{(n)}) 
           \end{align*}
			This completes the proof of sufficiency.\\
              The proof for necessity is on the similar lines as of  Gupta and Chaudhary (2023), hence omitted.\hfill $\blacksquare$
		\end{proof}

          The exponential distribution is significant in reliability theory and theoretical applications due to its memoryless property. Investigating various characterizations of the exponential distribution from different perspectives is fascinating. Qiu (2017) and Xiong, Zhuang, and Qiu (2021) demonstrated characterizations based on the extropy of order statistics and record values, respectively. In the following discussion, we will show that the exponential distribution can also be characterized through the weighted extropy of PRSS data.

            \begin{theorem}
               Let $X$ be a non negative absolutely continuous random variable with pdf $f$ and cdf $F$. Then $X$ has has standard exponential distribution if and only if, for all $n \geq 1$,
                \begin{align*}
			& J^{w_1}(\textbf{X}_{PRSS}^{(n)})\nonumber \\ 
 &=\begin{cases}
-\frac{Q_{1,n}(2n-1)!!}{2^{n+1}n^{n}} \frac{1}{\lambda^{n(m-1)}}\left(\prod_{i=1}^{n/2} E(W_{2a-1:2n}^{m})\right)\left(\prod_{i=\frac{n}{2}+1}^{n} E(W_{2b-1:2n}^{m}) \right) \text{ if } \mbox{n\  even},\\
-\frac{Q_{2,n}(2n-1)!!}{2^{n+1}n^{n-1}} \frac{1}{\lambda^{n(m-1)}}\left(\prod_{i=1}^{\frac{n-1}{2}} E(W_{2a-1:2n}^{m})\right) \left(\prod_{i=\frac{n+1}{2}}^{n-1}E(W_{2b-1:2n}^{m}) \right) \left( E(W_{n:2n}^{m})\right)\text{ if } \mbox{n\  odd}.
\end{cases}\nonumber \\
		\end{align*}
		where $W_{2i-1:2n}$ is the $(2i-1)$-th order statistics of a sample of size $2n$ from exponential distribution having pdf given by 
		$$\psi_{2i-1:2n}=\frac{(2n)!}{(2i-2)!(2n-2i+1)!}(1-e^{-x})^{2i-2}(e^{-x})^{2n-2i+2},\  x\ge 0,$$ 
		and  $(2n-1)!!= (2n-2a+1)^{\frac{n}{2}}(2n-2b+1)^{\frac{n}{2}} $ if $n$ is even and $(2n-1)!!= (2n-2a+1)^{\frac{n-1}{2}}(2n-2b+1)^{\frac{n-1}{2}} $if $n$ is odd, as obtained in Example 3.2. 
     \end{theorem}
     \begin{proof}
         The sufficiency follows from Example 3.2. Now we prove the necessity. Assume that equation,
  \begin{align*}
			& J^{w_1}(\textbf{X}_{PRSS}^{(n)})\nonumber \\ 
&=\begin{cases}
-\frac{Q_{1,n}(2n-1)!!}{2^{n+1}n^{n}} \frac{1}{\lambda^{n(m-1)}}\left(\prod_{i=1}^{n/2} E(W_{2a-1:2n}^{m})\right)\left(\prod_{i=\frac{n}{2}+1}^{n} E(W_{2b-1:2n}^{m}) \right) \text{ if } \mbox{n\  even},\\
-\frac{Q_{2,n}(2n-1)!!}{2^{n+1}n^{n-1}} \frac{1}{\lambda^{n(m-1)}}\left(\prod_{i=1}^{\frac{n-1}{2}} E(W_{2a-1:2n}^{m})\right) \left(\prod_{i=\frac{n+1}{2}}^{n-1}E(W_{2b-1:2n}^{m}) \right) \left( E(W_{n:2n}^{m})\right)\text{ if } \mbox{n\  odd}.
\end{cases}\nonumber 
		\end{align*}
 holds for all $n\geq 1$. In particular letting $n=1$ and $m=1$ we have,
     $ J^{w_1}(\textbf{X}_{PRSS}^{(1)})= \frac{-1}{8}$ 
     $=-\frac{1}{2}\int_{0}^{\infty}xe^{-2x} dx =J^w(X).$
     Since $ J^{w_1}(\textbf{X}_{PRSS}^{(1)})= J^w(X),$ This  implies that 
     $\int_{0}^{\infty} x (f(x)+e^{-x})(f(x)-e^{-x})dx=0.$
     which means $X$ has standard exponential distribution with $f(x)=e^{-x},\ x \geq 0.$  \hfill $\blacksquare$
    \end{proof}

		\section{Stochastic comparision}
		In the following result, we provide the conditions for comparing two PRSS schemes under different weights.
		
		\begin{theorem}\label{thm com rss1}
			Let $X$  and $Y$ be nonnegative random variables with pdf's $f$ and $g$, cdf's $F$ and $G$, respectively having $u_X=u_Y<\infty$.\\
			(a) If $w_1$ is increasing, $w_1(x)\geq w_2(x)$ and $X\le_{disp} Y$, then $J^{w_1}(\textbf{X}_{PRSS}^{(n)})\le J^{w_2}(\textbf{Y}_{PRSS}^{(n)})$.\\
			(b)   If $w_1$ is increasing, $w_1(x)\leq w_2(x)$ and $X\ge_{disp} Y$, then $J^{w_1}(\textbf{X}_{PRSS}^{(n)})\ge J^{w_2}(\textbf{Y}_{PRSS}^{(n)})$.
		\end{theorem}
		\begin{proof}
			(a) 
  The proof is on the similar lines as of Theroem 5.1 of Gupta and Chaudhary (2023), hence omitted.\\
(b) Proof is similar to part (a).\hfill $\blacksquare$
		\end{proof}
		
		If we take $w_1(x)=w_2(x)$ in the above theorem, then The following corollary follows.
		
		\begin{corollary}\label{cor com rss1}
			Let $X$  and $Y$ be nonnegative random variables with pdf's $f$ and $g$, cdf's $F$ and $G$, respectively having $u_X=u_Y<\infty$; let $w_1$ be increasing. Then\\
			(a) If $X\le_{disp} Y$, then $J^{w_1}(\textbf{X}_{PRSS}^{(n)})\le J^{w_1}(\textbf{Y}_{PRSS}^{(n)})$.\\
			(b)  If $X\ge_{disp} Y$, then $J^{w_1}(\textbf{X}_{PRSS}^{(n)})\ge J^{w_1}(\textbf{Y}_{PRSS}^{(n)})$.
		\end{corollary}
		

		\begin{lemma}\label{lemma1} [Ahmed et al. (1986); also see Qiu and Raqab (2022), lemma 4.3]
			Let $X$ and $Y$  be nonnegative random variables with pdf's $f$ and $g$, respectively, satisfying $f(0)\ge g(0)>0$. If $X\le_{su}Y$ (or $X\le_{*}Y$ or $X\le_{c}Y$), then $X\le_{disp}Y$.
		\end{lemma}
		
		One may refer Shaked and Shantikumar (2007) for details of convex transform order ($\leq_c$), star order ($\leq_{\star}$), super additive order ($\leq_{su}$), and dispersive order ($\leq_{disp}$). In view of Theorem \ref{thm com rss1}  and Lemma \ref{lemma1}, the following result is obtained.
		\begin{theorem}
			Let $X$  and $Y$ be nonnegative random variables with pdf's $f$ and $g$, cdf's $F$ and $G$, respectively having $u_X=u_Y<\infty$.\\
			(a)  If $w_1$ is increasing, $w_1(x)\geq w_2(x)$ and $X\le_{su} Y$ (or $X\le_{*}Y$ or $X\le_{c}Y$), then $J^{w_1}(\textbf{X}_{PRSS}^{(n)})\le J^{w_2}(\textbf{Y}_{PRSS}^{(n)})$.\\
			(b)   If $w_1$ is increasing, $w_1(x)\leq w_2(x)$ and $X\ge_{su} Y$  (or $X\ge_{*}Y$ or $X\ge_{c}Y$), then $J^{w_1}(\textbf{X}_{PRSS}^{(n)})\ge J^{w_2}(\textbf{Y}_{PRSS}^{(n)})$.
		\end{theorem}
		
		If we take $w_1(x)=w_2(x)$ in the above theorem, then we have the following corollary.
		\begin{corollary}\label{cor com rss2}
			Let $X$  and $Y$ be nonnegative random variables with pdf's $f$ and $g$, cdf's $F$ and $G$, respectively having $u_X=u_Y<\infty$ and  $w_1$ is increasing.\\
			(a) If $X\le_{su} Y$ (or $X\le_{*}Y$ or $X\le_{c}Y$), then $J^{w_1}(\textbf{X}_{PRSS}^{(n)})\le J^{w_1}(\textbf{Y}_{RSS}^{(n)})$.\\
			(b)  If $X\ge_{su} Y$  (or $X\ge_{*}Y$ or $X\ge_{c}Y$), then $J^{w_1}(\textbf{X}_{PRSS}^{(n)})\ge J^{w_1}(\textbf{Y}_{RSS}^{(n)})$.
		\end{corollary}
		

  One could question if the requirement $X\le_{disp}Y$ in Theorem \ref{thm com rss1} can be eased by the condition $J^{w_1}(X)\leq J^{w_2}(Y)$. The following outcome confirms this assertion.
		
		\begin{theorem}\label{thm com rss3}
			Let $X$  and $Y$ be nonnegative random variables with pdf's $f$ and $g$, cdf's $F$ and $G$, respectively. Let $\Delta (u)=w_1(F^{-1}(u))f(F^{-1}(u))-w_2(G^{-1}(u))g(G^{-1}(u))$,
			\[A_1=\{0\le u \le 1|\Delta (u)>0\},\ A_2=\{0\le u \le 1|\Delta (u)<0\}.\]
			If $ \inf_{A_1}\phi_{2i-1:2n-2i}(u)\geq  \sup_{A_2}\phi_{2i-1:2n-2i}(u)$, and if $J^{w_1}(X)\leq J^{w_2}(Y)$, then  $J^{w_1}(\textbf{X}_{PRSS}^{(n)})$ $\le J^{w_2}(\textbf{Y}_{PRSS}^{(n)})$.
		\end{theorem}
		\begin{proof}
   The proof is on the similar lines as of Theorem 5.3 of Gupta and Chaudhary (2023), hence omitted.\hfill $\blacksquare$ 
		\end{proof}
		
		If we take $w_1(x)=w_2(x)$ in the above theorem, then we have the following corollary.
		\begin{corollary}\label{cor com rss3}
			Let $X$  and $Y$ be nonnegative random variables with pdf's $f$ and $g$, cdf's $F$ and $G$. Let $\Delta (u)=w_1(F^{-1}(u))f(F^{-1}(u))-w_1(G^{-1}(u))g(G^{-1}(u))$,
			\[A_1=\{0\le u \le 1|\Delta (u)>0\},\ A_2=\{0\le u \le 1|\Delta (u)<0\}.\]
			If $\ \inf_{A_1}\phi_{2i-1:2n-2i}(u) \geq \sup_{A_2}\phi_{2i-1:2n-2i}(u)$ , and if $J^{w_1}(X)\leq J^{w_1}(Y)$, then  $J^{w_1}(\textbf{X}_{PRSS}^{(n)})$ $\le J^{w_1}(\textbf{Y}_{PRSS}^{(n)})$.
		\end{corollary}
		
		
		\begin{example}
			Let $X$  and $Y$ be nonnegative random variables with pdf's $f$ and $g$, respectively. Let
			\begin{eqnarray*}
				f(x)=
				\begin{cases}
					2x, \ 0\leq x<1\\
					0, otherwise
				\end{cases} \ \ ~~~~~~~\text{and}~~~~~~~~~~~~~~~~~
				\ g(x)=
				\begin{cases}
					2(1-x),\  0\leq x<1\\
					0, otherwise.
				\end{cases}
			\end{eqnarray*}
			As pointed in Qiu and Raqab (2022), Example 4.10, $ X \nleq_{disp}Y$ and $ \inf_{A_1}\phi_{2i-1:2n-2i}(u) = \sup_{A_2}\phi_{2i-1:2n-2i}(u)$, where $ A_{1}= (0,1] $, $ A_{2} = \{ 0\} $. If we choose $w_1(x)=x^2$ and $w_2(x)=x$, then $J^{w_1}(X) < J^{w_2}(Y)$ (see Gupta and Chaudhary (2023), Example 5.1). Hence using Theorem \ref{thm com rss3}, we have  $J^{w_1}(\textbf{X}_{PRSS}^{(n)})$ $\le J^{w_2}(\textbf{Y}_{PRSS}^{(n)})$. 
		\end{example}

		
         
		\textbf{ \Large Conflict of interest} \\
		\\
		No conflicts of interest are disclosed by the authors.\\
		\\
		\\		
		\textbf{ \Large Funding} \\
		\\
		PKS would like to thank Quality Improvement Program (QIP), All India Council for Technical Education, Government of India (Student Unique Id: FP2200759) for financial assistance. \\

		
		

	\end{document}